\documentclass{amsart}
\usepackage{amssymb}
\usepackage{amsmath}
\usepackage{amsfonts}

\newtheorem{theorem}{Theorem}
\theoremstyle{plain}

\newtheorem{corollary}{Corollary}

\newtheorem{lemma}{Lemma}

\newtheorem{proposition}{Proposition}

\numberwithin{equation}{section}
\input{tcilatex}

\begin{document}
\title[On metric connections with torsion on the cotangent bundle]{On metric
connections with torsion on the cotangent bundle with modified Riemannian
extension}
\author{Lokman BILEN}
\address{Igdir University, Faculty of Science and Letters, Department of
Mathematics and Computer, 76000, Igdir-Turkey.}
\email{lokman.bilen@igdir.edu.tr}
\author{Aydin GEZER}
\address{Ataturk University, Faculty of Science, Department of Mathematics,
25240, Erzurum-Turkey.}
\email{agezer@atauni.edu.tr}

\begin{abstract}
Let $M$ be an $n-$dimensional differentiable manifold equipped with a
torsion-free linear connection $\nabla $ and $T^{\ast }M$ its cotangent
bundle. The present paper aims to study a metric connection $\widetilde{%
\nabla }$ with nonvanishing torsion on $T^{\ast }M$ with modified Riemannian
extension ${}\overline{g}_{\nabla ,c}$. First, we give a characterization of
fibre-preserving projective vector fields on $(T^{\ast }M,{}\overline{g}%
_{\nabla ,c})$ with respect to the metric connection $\widetilde{\nabla }$.
Secondly, we study conditions for $(T^{\ast }M,{}\overline{g}_{\nabla ,c})$
to be semi-symmetric, Ricci semi-symmetric, $\widetilde{Z}$ semi-symmetric
or locally conharmonically flat with respect to the metric connection $%
\widetilde{\nabla }$. Finally, we present some results concerning the
Schouten-Van Kampen connection associated to the Levi-Civita connection $%
\overline{\nabla }$ of the modified Riemannian extension $\overline{g}%
_{\nabla ,c}$.


\end{abstract}

\maketitle

\section{\protect\bigskip \textbf{Introduction}}

Let $(M,\nabla )$ be an $n-$dimensional differentiable manifold equipped
with a torsion-free linear connection $\nabla $. Denote by $T^{\ast }M$ the
cotangent bundle of $M$ and let $\pi $ be the natural projection $T^{\ast
}M\rightarrow M$. The vertical distribution $V$ on $T^{\ast }M$ ($V$ is the
kernel of the submersion $T^{\ast }M\rightarrow M$), which is the integrable
distribution. If $M$ is a paracompact manifold there exists a $C^{\infty }-$%
distribution $H$ on $T^{\ast }M$ which is supplementary to the vertical
distribution $V$, such as the Whitney sum $TT^{\ast }M=HT^{\ast }M\oplus
VT^{\ast }M$ holds.

For the torsion-free linear connection $\nabla $ on $M$, the cotangent
bundle of $M$, $T^{\ast }M$, can be endowed with a pseudo-Riemannian metric $%
\overline{g}_{\nabla }$ of neutral signature, called the Riemannian
extension of $\nabla $, given by%
\begin{eqnarray*}
{}\overline{g}_{\nabla }(^{H}X,^{H}Y) &=&0 \\
{}\overline{g}_{\nabla }(^{H}X,^{V}\omega ) &=&\overline{g}_{\nabla
}(^{V}\omega ,^{H}X)=\omega (X) \\
{}\overline{g}_{\nabla }(^{V}\omega ,^{V}\theta ) &=&0
\end{eqnarray*}%
where $^{H}X$ and $^{H}Y$ denote the horizontal lifts of the vector fields $%
X $ and $Y$, and $^{V}\omega $ and $^{V}\theta $ denote the vertical lifts
of the covectors ($1-$forms) $\omega $ and $\theta $. Thus, the Riemannian
extension of $(M,\nabla )$ is a pseudo-Riemannian manifold $(T^{\ast }M,%
\overline{g}_{\nabla })$. Riemannian extensions were first defined and
studied by Patterson and Walker \cite{Patterson} and then investigated in
Afifi \cite{Afifi}. Moreover, Riemannian extensions were also considered in
Garcia-Rio et al. \cite{Garcia1} in relation to Osserman manifolds (see also
Derdzinski \cite{Derdzinski}). For further references relation to Riemannian
extensions, see \cite{Aslanci,Dryuma,Honda,Mok,Toomanian,Vanhecke,Willmore}.
Classical Riemannian extensions have been generalized in several ways, see,
as an example \cite{Kowalski}. In \cite{Calvino1,Calvino2}, the authors
introduced another generalization which is called modified Riemannian
extension. For a symmetric $(0,2)$-tensor field $c$ on $(M,\nabla )$, this
metric is given by $\overline{g}_{\nabla ,c}=\overline{g}_{\nabla }+\pi
^{\ast }c$, that is, 
\begin{eqnarray*}
{}\overline{g}_{\nabla ,c}(^{H}X,^{H}Y) &=&c(X,Y) \\
{}\overline{g}_{\nabla ,c}(^{H}X,^{V}\omega ) &=&\overline{g}_{\nabla
,c}(^{V}\omega ,^{H}X)=\omega (X) \\
{}\overline{g}_{\nabla ,c}(^{V}\omega ,^{V}\theta ) &=&0.
\end{eqnarray*}

In this paper, we consider a metric connection $\widetilde{\nabla }$ with
nonvanishing torsion on the cotangent bundle $T^{\ast }M$ with modified
Riemannian extension ${}\overline{g}_{\nabla ,c}$. First, we give a
necessary and sufficient condition for a vector field on $(T^{\ast }M,%
\overline{g}_{\nabla ,c})$ to be fibre-preserving projective vector field on 
$T^{\ast }M$ with respect to the metric connection $\widetilde{\nabla }$.
This condition is represented by a set of relations involving certain tensor
fields on $M$. Secondly, we investigate the conditions for the cotangent
bundle $(T^{\ast }M,\overline{g}_{\nabla ,c})$ to be semi-symmetric, Ricci
semi-symmetric, $\widetilde{Z}$ semi-symmetric and locally conharmonically
flat with respect to the metric connection $\widetilde{\nabla }$. Finally,
we show that the Schouten-Van Kampen connection associated to the
Levi-Civita connection $\overline{\nabla }$ of the modified Riemannian
extension $\overline{g}_{\nabla ,c}$ is equal to the horizontal lift $%
^{H}\nabla $ of the torsion-free linear connection $\nabla $ to $T^{\ast }M$
and present a result concerning the curvature tensor of the Schouten-Van
Kampen connection.

The manifolds, tensor fields and geometric objects we consider in this paper
are assumed to be differentiable of class $C^{\infty }$. Einstein's
summation convention is used, the range of the indices $h,i,j,k,l,m,r,$
being always $\{1,...,n\}$.

\section{\noindent Preliminaries}

\noindent We refer to \cite{YanoIshihara:DiffGeo} for further details
concerning the material of this section. Let $M$ be an $n-$dimensional
differentiable manifold with a torsion-free linear connection $\nabla $ and
denote by $\pi :T^{\ast }M\rightarrow M$ its cotangent bundle with fibres
the cotangent spaces to $M$. Then ${}T^{\ast }M$ is a $2n-$dimensional
smooth manifold and some local charts induced naturally from local charts on 
$M$, may be used. Namely, a system of local coordinates $\left(
U,x^{i}\right) ,\mathrm{\;}i=1,...,n$ on $M$ induces on ${}T^{\ast }M$ a
system of local coordinates $\left( \pi ^{-1}\left( U\right) ,\mathrm{\;}%
x^{i},\mathrm{\;}x^{\overline{i}}=p_{i}\right) ,\mathrm{\;}\overline{i}%
=n+i=n+1,...,2n$, where $x^{\overline{i}}=p_{i}$ is the components of
covectors $p$ in each cotangent space ${}T_{x}^{\ast }M,\mathrm{\;}x\in U$
with respect to the natural coframe $\left\{ dx^{i}\right\} $.

Let $X=X^{i}\frac{\partial }{\partial x^{i}}$ and $\omega =\omega _{i}dx^{i}$
be the local expressions in $U$ of a vector field $X$ \ and a covector field 
$\omega $ on $M$, respectively. Then the vertical lift $^{V}\omega $ of $%
\omega $, the horizontal lift $^{H}X$ of $X$ are given, with respect to the
induced coordinates, by\noindent 
\begin{equation*}
^{V}\omega =\omega _{i}\partial _{\overline{i}},
\end{equation*}%
and

\begin{equation*}
^{H}X=X^{i}\partial _{i}+p_{h}\Gamma _{ij}^{h}X^{j}\partial _{\overline{i}}
\end{equation*}%
where $\partial _{i}=\frac{\partial }{\partial x^{i}}$, $\partial _{%
\overline{i}}=\frac{\partial }{\partial x^{\overline{i}}}$ and $\Gamma
_{ij}^{h}$ are the coefficients of $\nabla $ on $M$.

Next, we can introduce a frame field on each induced coordinate neighborhood 
$\pi ^{-1}(U)$ of $T^{\ast }M$. It is called the adapted frame and consists
of the following $2n$ linearly independent vector fields $\left\{ E_{\beta
}\right\} =\left\{ E_{j},E_{\overline{j}}\right\} $:%
\begin{equation*}
\left\{ 
\begin{array}{c}
E_{j}=\text{ }\partial _{j}+p_{a}\Gamma _{hj}^{a}\partial _{\overline{h}} \\ 
E_{\overline{j}}=\text{ }\partial _{\overline{j}}.%
\end{array}%
\right.
\end{equation*}%
The indices $\alpha ,\beta ,\gamma ,...=1,...,2n$ indicate the indices with
respect to the adapted frame. The Lie brackets of the adapted frame of $%
T^{\ast }M$ satisfy the following identities:%
\begin{equation}
\left\{ 
\begin{array}{c}
\left[ E_{i},E_{j}\right] =p_{s}R_{ijl}^{\text{ \ \ }s}E_{\overline{l}}, \\ 
\left[ E_{i},E_{\overline{j}}\right] =-\Gamma _{il}^{j}E_{\overline{l}}, \\ 
\left[ E_{\overline{i}},E_{\overline{j}}\right] =0,%
\end{array}%
\right.  \label{BA2.3}
\end{equation}%
where $R_{ijl}^{\text{ \ \ }s}$ denote the coefficients of the curvature
tensor $R$ of $\nabla $ on $M$.\bigskip

With respect to the adapted frame $\left\{ E_{\beta }\right\} $, the vector
fields $^{V}\omega $ and $^{H}X$ on $T^{\ast }M$ has the components%
\begin{equation*}
^{V}\omega =\left( 
\begin{array}{l}
0 \\ 
\omega _{j}%
\end{array}%
\right) \text{ and }^{H}X=\left( 
\begin{array}{l}
X^{j} \\ 
0%
\end{array}%
\right) .
\end{equation*}%
\bigskip

\section{The metric connection with nonvanishing torsion on the cotangent
bundle with respect to modified Riemannian extension}

Let us consider $T^{\ast }M$ equipped with the modified Riemannian extension 
$\overline{g}_{\nabla ,c}$ for a given torsion-free connection $\nabla $ on $%
M$. In adapted frame $\left\{ E_{\beta }\right\} $, the modified Riemannian
extension $(\overline{g}_{\nabla ,c})_{\beta \gamma }$ and its inverse $(%
\overline{g}_{\nabla ,c})^{\beta \gamma }$ have in the following forms:%
\begin{equation}
(\overline{g}_{\nabla ,c})_{\beta \gamma }=\left( 
\begin{array}{cc}
{c}_{ij} & {\delta }_{i}^{j} \\ 
{\delta }_{j}^{i} & {0}%
\end{array}%
\right) .\text{ }  \label{BA3.1}
\end{equation}%
\begin{equation}
(\overline{g}_{\nabla ,c})^{\beta \gamma }=\left( 
\begin{array}{cc}
{0} & {\delta }_{j}^{i} \\ 
{\delta }_{i}^{j} & -{c}_{ij}%
\end{array}%
\right)  \label{BA3.2}
\end{equation}%
where $c_{ij}$ are the components of the symmetric $(0,2)-$tensor field $c$
on $(M,\nabla )$.

For the Levi-Civita connection $\overline{\nabla }$ of the modified
Riemannian extension $\overline{g}_{\nabla ,c}$, we get:

\begin{proposition}
Let $\nabla $ be a torsion-free linear connection on $M$ and $T^{\ast }M$ be
the cotangent bundle with the modified Riemann extension $\overline{g}%
_{\nabla ,c}$ over $(M,\nabla )$. The Levi-Civita connection $\overline{%
\nabla }$ of $(T^{\ast }M,\overline{g}_{\nabla ,c})$ is given by%
\begin{equation}
\left\{ 
\begin{array}{c}
\overline{\nabla }_{E_{\overline{i}}}E_{\overline{j}}=0,\text{ }\overline{%
\nabla }_{E_{\overline{i}}}E_{j}=0, \\ 
\overline{\nabla }_{E_{i}}E_{\overline{j}}=-\Gamma _{ih}^{j}E_{\overline{h}},
\\ 
\overline{\nabla }_{E_{i}}E_{j}=\Gamma _{ij}^{h}E_{h}+\{p_{s}R_{hji}^{\text{
\ \ \ \ }s}+\frac{1}{2}(\nabla _{i}c_{jh}+\nabla _{j}c_{ih}-\nabla
_{h}c_{ij})\}E_{\overline{h}}%
\end{array}%
\right.  \label{BA3.3}
\end{equation}%
with respect to the adapted frame $\left\{ E_{\beta }\right\} $, where $%
\Gamma _{ij}^{h}$ and $R_{hji}^{\text{ \ \ \ \ }s}$ respectively denote
components of $\nabla $ and its curvature tensor field $R$ on $M$ (see, \cite%
{Gezer1}).
\end{proposition}

If there is a Riemannian metric $g$ on $M$ such that $\nabla g=0$, then the
connection $\nabla $ is a metric connection, otherwise it is non-metric. It
is well known that a linear connection is symmetric and metric if and only
if it is the Levi-Civita connection. The Levi-Civita connection $\overline{%
\nabla }$ of the modified Riemannian extension $\overline{g}_{\nabla ,c}$ on 
$T^{\ast }M$ is the unique connection which satisfies $\overline{\nabla }%
_{\alpha }(\overline{g}_{\nabla ,c})_{\beta \gamma }=0$ and has a zero
torsion. Now we are interested in a metric connection $\widetilde{\nabla }$
of the modified Riemannian extension $\overline{g}_{\nabla ,c}$ whose
torsion tensor $\widetilde{T}_{\gamma \beta }^{\varepsilon }$ is
skew-symmetric in the indices $\gamma $ and $\beta $. Metric connection with
nonvanishing torsion on Riemannian manifolds were introduced by Hayden \cite%
{Hayden}. We denote components of the metric connection $\widetilde{\nabla }$
by $\widetilde{\Gamma }_{\alpha \beta }^{\gamma }$. The metric connection $%
\widetilde{\nabla }$ satisfies 
\begin{equation*}
\widetilde{\nabla }_{\alpha }(\overline{g}_{\nabla ,c})_{\beta \gamma }=0%
\text{ and }\widetilde{\Gamma }_{\alpha \beta }^{\gamma }-\text{ }\widetilde{%
\Gamma }_{\beta \alpha }^{\gamma }=\text{ }\widetilde{T}_{\alpha \beta
}^{\gamma }.
\end{equation*}%
When the above equation is solved with respect to $\widetilde{\Gamma }%
_{\alpha \beta }^{\gamma }$, one finds the following solution \cite{Hayden}%
\begin{equation}
\widetilde{\Gamma }_{\alpha \beta }^{\gamma }=\overline{\Gamma }_{\alpha
\beta }^{\gamma }+\widetilde{U}_{\alpha \beta }^{\gamma },  \label{BA3.4}
\end{equation}%
where $\overline{\Gamma }_{\alpha \beta }^{\gamma }$ is the components of
the Levi-Civita connection $\overline{\nabla }$ of the modified Riemannian
extension $\overline{g}_{\nabla ,c}$, 
\begin{equation}
\widetilde{U}_{\alpha \beta \gamma }=\frac{1}{2}(\widetilde{T}_{\alpha \beta
\gamma }+\widetilde{T}_{\gamma \alpha \beta }+\widetilde{T}_{\gamma \beta
\alpha })  \label{BA3.5}
\end{equation}%
and%
\begin{equation*}
\widetilde{U}_{\alpha \beta \gamma }=\widetilde{U}_{\alpha \beta }^{\epsilon
}(\overline{g}_{\nabla ,c})_{\epsilon \gamma },\text{ }\widetilde{T}_{\alpha
\beta \gamma }=\widetilde{T}_{\alpha \beta }^{\epsilon }(\overline{g}%
_{\nabla ,c})_{\epsilon \gamma }.
\end{equation*}

If we choose the torsion tensor $\widetilde{T}$ as 
\begin{equation}
\left\{ 
\begin{array}{c}
\widetilde{T}_{ij}^{\overline{r}}=-p_{s}R_{ijr}^{\text{ \ \ }s}, \\ 
otherwise=0,%
\end{array}%
\right.  \label{BA3.6}
\end{equation}%
with the help of (\ref{BA3.6}), from (\ref{BA3.5}), we find non-zero
component of $\widetilde{U}_{\alpha \beta }^{\gamma }$ as follows: 
\begin{equation*}
\widetilde{{U}}{_{ij}^{\overline{h}}}{=}p_{s}R_{jhi}^{\text{ \ \ }s}
\end{equation*}%
with respect to the adapted frame. In view of (\ref{BA3.3}) and (\ref{BA3.4}%
), we have the following proposition.

\begin{proposition}
\label{propo0}Let $\nabla $ be a torsion-free linear connection on $M$ and $%
T^{\ast }M$ be the cotangent bundle with the modified Riemann extension $%
\overline{g}_{\nabla ,c}$ over $(M,\nabla )$. The metric connection $%
\widetilde{\nabla }$ on $T^{\ast }M$ with respect to the modified Riemannian
extension $\overline{g}_{\nabla ,c}$ satisfies%
\begin{equation}
\left\{ 
\begin{array}{c}
\widetilde{\nabla }_{E_{\overline{i}}}E_{\overline{j}}=0,\text{ }\widetilde{%
\nabla }_{E_{\overline{i}}}E_{j}=0, \\ 
\widetilde{\nabla }_{E_{i}}E_{\overline{j}}=-\Gamma _{ih}^{j}E_{\overline{h}%
}, \\ 
\widetilde{\nabla }_{E_{i}}E_{j}=\Gamma _{ij}^{h}E_{h}+\frac{1}{2}(\nabla
_{i}c_{jh}+\nabla _{j}c_{ih}-\nabla _{h}c_{ij})E_{\overline{h}}%
\end{array}%
\right.  \label{BA3.7}
\end{equation}%
with respect to the adapted frame $\left\{ E_{\beta }\right\} $.
\end{proposition}

The horizontal lift $^{H}\nabla $ of the torsion-free linear connection $%
\nabla $ on $M$ to $T^{\ast }M$ is characterized the following conditions:%
\begin{equation*}
\left\{ 
\begin{array}{c}
^{H}\nabla _{^{V}\omega }\text{ }^{V}\theta =0,\text{ }^{H}\nabla
_{^{V}\omega }\text{ }^{H}Y=0 \\ 
^{H}\nabla _{^{H}X}\text{ }^{V}\theta =\text{ }^{V}(\nabla _{X}\text{ }%
\theta ),\text{ }^{H}\nabla _{^{H}X}\text{ }^{H}Y=\text{ }^{H}(\nabla _{X}Y)%
\end{array}%
\right.
\end{equation*}%
for all vector fields $X,Y$ and covector fields $\omega ,\theta $ on $M$ (%
\cite{YanoIshihara:DiffGeo}, p. 287). In the adapted frame, the followings
satisfy (see, also \cite{Aslanci}) 
\begin{equation*}
\left\{ 
\begin{array}{c}
^{H}\nabla _{E_{\overline{i}}}E_{\overline{j}}=0,\text{ }^{H}\nabla _{E_{%
\overline{i}}}E_{j}=0, \\ 
^{H}\nabla _{E_{i}}E_{\overline{j}}=-\Gamma _{ih}^{j}E_{\overline{h}},\text{ 
}^{H}\nabla _{E_{i}}E_{j}=\Gamma _{ij}^{h}E_{h}.%
\end{array}%
\right.
\end{equation*}%
From these formulas, we can readily deduce:

\begin{proposition}
\label{propo1} Let $\nabla $ be a torsion-free linear connection on $M$ and $%
T^{\ast }M$ be the cotangent bundle with the modified Riemann extension $%
\overline{g}_{\nabla ,c}$ over $(M,\nabla )$. The metric connection ${%
\widetilde{{}\nabla }}$ on $T^{\ast }M$ of the modified Riemannian extension 
$\overline{g}_{\nabla ,c}$ coincides with the horizontal lift $^{H}\nabla $
of the torsion-free linear connection $\nabla $ on $M$ if and only if the
components $c_{ij}$ of $c$ satisfy the condition 
\begin{equation*}
{\nabla _{i}c_{jh}+\nabla _{j}c_{ih}-\nabla _{h}c_{ij}=0}.
\end{equation*}
\end{proposition}

\subsection{Projective vector fields on the cotangent bundle with respect to
the metric connection ${\protect\widetilde{{}\protect\nabla }}$}

Given a linear connection $\nabla $ on a manifold $M$, a vector field $V$ is
said to be a projective vector field if there exists a $1-$form $\theta $
such that 
\begin{equation*}
(L_{V}\nabla )(X,Y)=\theta (X)Y+\theta (Y)X
\end{equation*}%
for any pair of vector fields $X$ and $Y$ on $M$. In particular, if $\theta
=0$, $V$ is an affine Killing vector field.

Let $\widetilde{V}$ be a vector field on $T^{\ast }M$ and $(v^{h},v^{%
\overline{h}})$ its the components with respect to the adapted frame $%
\left\{ E_{\beta }\right\} $. The components $v^{h}$ and $v^{\overline{h}}$
are said to be the horizontal components and vertical components of $%
\widetilde{V}$, respectively. As is known, a vector field is called a
fibre-preserving vector field if and only if its horizontal components
depend only on the variables $(x^{h})$. Hence, one can easily say that every
fibre-preserving vector field $\widetilde{V}$ on $T^{\ast }M$ induces a
vector field $V$ with components $(v^{h})$ on the base manifold $M$.

By straightforward calculations, we have the following.

\begin{lemma}
\label{lemma1} Let $\widetilde{V}$ be a fibre-preserving vector field on $%
T^{\ast }M$ with components $(v^{h},v^{\overline{h}})$. The Lie derivatives
of the adapted frame satisfy%
\begin{eqnarray*}
i)L_{\widetilde{V}}E_{i} &=&-\left( E_{i}v^{k}\right) E_{k}-\left(
v^{a}p_{s}R_{iak}^{\mathrm{\;\;\;}s}+E_{i}v^{\overline{k}}-v^{\overline{a}%
}\Gamma _{ik}^{\mathrm{\;\;}a}\right) E_{\overline{k}}, \\
ii)L_{\widetilde{V}}E_{\overline{i}} &=&-\left( v^{a}\Gamma _{ak}^{\mathrm{%
\;\;}i}+E_{\overline{i}}v^{\overline{k}}\right) E_{\overline{k}},
\end{eqnarray*}%
where $L_{\widetilde{V}}$ denotes the Lie derivation with respect to $%
\widetilde{V}$.
\end{lemma}

The general forms of fibre-preserving projective vector fields on $T^{\ast
}M $ with respect to the metric connection $\widetilde{\nabla }$ are given
by:

\begin{theorem}
\ \label{theo3}Let $\nabla $ be a torsion-free linear connection on $M$ and $%
T^{\ast }M$ be the cotangent bundle with the modified Riemannian extension $%
\overline{g}_{\nabla ,c}$ over $(M,\nabla )$. Then a vector field $\;%
\widetilde{V}$ is a fibre-preserving projective vector field with associated 
$1-$form $\widetilde{\theta }$ on $T^{\ast }M$ with respect to the metric
connection $\widetilde{\nabla }$ if and only if the vector field $\widetilde{%
V}$ is defined by 
\begin{equation}
\widetilde{X}={}^{H}V+^{V}B+\gamma A,  \label{BA31.1}
\end{equation}%
where the vector field $V=(v^{h}),$ the covector field $B=(B_{h})$, the $%
(1,1)-$tensor field $A=(A_{i}^{h})$ and the associated $1-$form $\widetilde{%
\theta }$ satisfy

$(i)$ $\widetilde{\theta }=\theta _{i}dx^{i},$ $\ $

$(ii)$ $\nabla _{j}A_{k}^{i}=\theta _{j}\delta _{i}^{k}-v^{a}R_{jak}^{%
\mathrm{\;\;\;}i},$

$(iii)$ $L_{V}\Gamma _{ij}^{h}=\theta _{i}\delta _{j}^{h}+\theta _{j}\delta
_{i}^{h}$

$(iv)$ $\nabla _{i}\nabla _{j}B_{l}+R_{lji}^{\text{ \ \ }a}B_{a}+\frac{1}{2}%
v^{a}\nabla _{a}M_{ijl}+\frac{1}{2}(\nabla _{j}v^{a})M_{ial}$

$+\frac{1}{2}(\nabla _{i}v^{a})M_{ajl}-A_{l}^{a}M_{ija}=0$ \ $(M_{ijl}:={%
\nabla _{i}c_{jl}+\nabla _{j}c_{il}-\nabla _{l}c_{ij})}$

$(v)$ $\nabla _{i}\nabla _{j}A_{l}^{s}+R_{lji}^{\text{ \ \ }%
a}A_{a}^{s}-R_{aji}^{\text{ \ \ }s}A_{l}^{a}+v^{a}\nabla _{i}R_{jal}^{\text{
\ \ }s}+(\nabla _{i}v^{a})R_{jal}^{\text{ \ \ }s}=0.$
\end{theorem}

\begin{proof}
\noindent A fibre-preserving vector field $\widetilde{V}=v^{h}E_{h}+v^{%
\overline{h}}E_{\overline{h}}$ on $T^{\ast }M$ is a fibre-preserving
projective vector field if and only if there exist a $1-$form $\widetilde{%
\theta }$ with components $(\widetilde{\theta }_{i},\widetilde{\theta }_{%
\overline{i}})$ on $T^{\ast }M$ such that 
\begin{eqnarray}
(L_{\widetilde{V}}^{\text{ \ \ }}\widetilde{\nabla })(\widetilde{Y},%
\widetilde{Z}) &=&L_{\widetilde{V}}({}\widetilde{\nabla }_{\widetilde{Y}}%
\widetilde{Z})-{}\widetilde{\nabla }_{\widetilde{Y}}(L_{\widetilde{V}}%
\widetilde{Z})-{}\widetilde{\nabla }_{(L_{\widetilde{V}}\widetilde{Y})}%
\widetilde{Z}  \label{BA31.2} \\
&=&\widetilde{\theta }(\widetilde{Y})\widetilde{Z}+\widetilde{\theta }(%
\widetilde{Z})\widetilde{Y}  \notag
\end{eqnarray}%
for any vector fields $\tilde{Y}$ and $\widetilde{Z}$ on $T^{\ast }M$.

Putting $\widetilde{Y}=E_{\overline{i}},\widetilde{Z}=E_{\bar{j}}$ in (\ref%
{BA31.2}), we get%
\begin{equation}
E_{\overline{i}}\left( E_{\bar{j}}v^{\bar{k}}\right) E_{\bar{k}}=\theta _{%
\overline{i}}E_{\bar{j}}+\theta _{\bar{j}}E_{\overline{i}}.  \label{BA31.3}
\end{equation}

Putting $Y=E_{\overline{i}},Z=E_{j}$ in (\ref{BA31.2}), we find 
\begin{equation}
\theta _{\overline{i}}=0  \label{BA31.4}
\end{equation}%
and%
\begin{equation}
v^{a}R_{jak}^{\mathrm{\;\;\;}i}+E_{\overline{i}}\left( E_{j}v^{\bar{k}%
}\right) -\left( E_{\overline{i}}v^{\bar{a}}\right) \Gamma _{jk}^{\mathrm{%
\;\;}a}=\theta _{j}\delta _{i}^{k}.  \label{BA31.5}
\end{equation}

In view of (\ref{BA31.4}), (\ref{BA31.3}) reduces to

\begin{equation*}
E_{\overline{i}}\left( E_{\bar{j}}v^{\bar{k}}\right) E_{\bar{k}}=0
\end{equation*}%
from which it follows that%
\begin{equation}
v^{\bar{k}}=p_{s}A_{k}^{s}+B_{k}  \label{BA31.6}
\end{equation}%
where $A_{k}^{s}$ and $B_{k}$ are certain functions which depend only on the
variables $(x^{h})$. The coordinate transformation rule implies that $A$ is
a $(1,1)-$tensor field with components $(A_{k}^{s})$ and $B$ is a covector
field with components $(B_{k})$. Hence, the fibre-preserving projective
vector field $\widetilde{V}$ on $T^{\ast }M$ can be written in the form:%
\begin{eqnarray*}
\widetilde{V} &=&v^{k}E_{k}+v^{\overline{k}}E_{\overline{k}%
}=v^{k}E_{k}+\{p_{s}A_{k}^{s}+B_{k}\}E_{\overline{k}} \\
&=&{}^{H}V+^{V}B+\gamma A
\end{eqnarray*}%
where $\gamma $ is an operator applied to the $(1,1)-$tensor field $A$ and
expressed locally $\gamma A=(p_{s}A_{k}^{s})E_{\overline{k}}$ (for details
related to the operator $\gamma $, see \cite{YanoIshihara:DiffGeo}, p.$12-13$%
).

Substitution (\ref{BA31.6}) into (\ref{BA31.5}) gives 
\begin{equation}
v^{a}R_{jak}^{\mathrm{\;\;\;}i}+\nabla _{j}A_{k}^{i}=\theta _{j}\delta
_{i}^{k}.  \label{BA31.7}
\end{equation}%
Contracting $i$ and $k$ in (\ref{BA31.7}), we have%
\begin{equation*}
\theta _{j}=\frac{1}{n}\nabla _{j}A_{k}^{k}.
\end{equation*}

Finally, putting $Y=E_{i},Z=E_{j}$ in (\ref{BA31.2}), we obtain

\begin{equation*}
L_{V}\Gamma _{ij}^{h}=\theta _{i}\delta _{j}^{h}+\theta _{j}\delta _{i}^{h},
\end{equation*}%
\begin{eqnarray*}
&&\nabla _{i}\nabla _{j}B_{l}-R_{ijl}^{\text{ \ \ }a}B_{a}+\frac{1}{2}%
v^{a}\nabla _{a}M_{ijl} \\
&&+\frac{1}{2}(\nabla _{j}v^{a})M_{ial}+\frac{1}{2}(\nabla
_{i}v^{a})M_{ajl}-A_{l}^{a}M_{ija} \\
&=&0
\end{eqnarray*}%
and%
\begin{equation*}
\nabla _{i}\nabla _{j}A_{l}^{s}+R_{lji}^{\text{ \ \ }a}A_{a}^{s}-R_{aji}^{%
\text{ \ \ }s}A_{l}^{a}+v^{a}\nabla _{i}R_{jal}^{\text{ \ \ }s}+(\nabla
_{i}v^{a})R_{jal}^{\text{ \ \ }s}=0,
\end{equation*}%
where $M_{ijl}={\nabla _{i}c_{jl}+\nabla _{j}c_{il}-\nabla _{l}c_{ij}}$.

Conversely, if $B_{h},$ $v^{h},\theta _{h}$ and $A_{i}^{h}$ are given so
that they satisfy $(i)-(v)$, reversing the above steps, we see that $\;%
\widetilde{V}={}^{H}V+^{V}B+\gamma A$ is a fibre-preserving projective
vector field on $T^{\ast }M$ with respect to the metric connection $%
\widetilde{\nabla }$. This completes the proof.
\end{proof}

The below result follows immediately from Theorem \ref{theo3} and its Proof.

\begin{corollary}
Let $\nabla $ be a torsion-free linear connection on $M$ and $T^{\ast }M$ be
the cotangent bundle with the modified Riemannian extension $\overline{g}%
_{\nabla ,c}$ over $(M,\nabla )$. Every fibre-preserving projective vector
field $\widetilde{V}$ with respect to the metric connection $\widetilde{%
\nabla }$ is of the form (\ref{BA31.1}) and it naturally induces a
projective vector field $V$ on $M$.
\end{corollary}

\subsection{Semi-Symmetry properties of the cotangent bundle with respect to
the metric connection $\protect\widetilde{\protect\nabla }$}

Given a manifold $M$ $(\dim (M)\geq 3)$ endowed with a linear connection $%
\nabla $ whose curvature tensor is marked as $R$, for any tensor field of $S$
of type $(0,k),k\geq 1,$ the tensor field $R(X,Y).S$ is expressed in the
form:%
\begin{eqnarray*}
(R(X,Y).S)(X_{1},X_{2},...,X_{k}) &=&-S(R(X,Y)X_{1},X_{2},...,X_{k}) \\
&&-...-S(X_{1},X_{2},...,X_{k-1},R(X,Y)X_{k})
\end{eqnarray*}%
for any vector fields $X_{1},X_{2},...,X_{k},X,Y$ on $M$, where $R(X,Y)$
acts as a derivation on $S$. If $R(X,Y).S=0$, then the manifold $M$ is said
to be $S$ semi-symmetric with respect to the linear connection $\nabla $. A
(pseudo-) Riemannian manifold $(M,g)$ such that its curvature tensor $R$
satisfies the condition 
\begin{equation*}
R(X,Y).R=0
\end{equation*}%
is called a semi-symmetric space. Also, note that locally symmetric spaces
are semi-symmetric, but in general the converse is not true. The
semi-symmetric space was first studied by Cartan. Nevertheless, Sinjukov
first used the name "semi-symmetric spaces" for manifolds satisfying the
above curvature condition \cite{Sinjukov}. Later, Szabo gave the full local
and global classification of semi-symmetric spaces \cite{Szabo1,Szabo2}. A
(pseudo-)Riemannian manifold $(M,g)$ is called Ricci semi-symmetric if the
following condition is satisfied:%
\begin{equation*}
R(X,Y).Ric=0,
\end{equation*}%
where $Ric$ is the Ricci tensor of $(M,g)$. It is obvious that any
semi-symmetric manifold is Ricci semi-symmetric.

The curvature tensor $\widetilde{R}$ of the metric connection $\widetilde{%
\nabla }$ on $T^{\ast }M$ is obtained from the formula%
\begin{equation*}
\widetilde{R}(E_{\alpha },E_{\beta })E_{\gamma }=\widetilde{\nabla }%
_{E_{\alpha }}\widetilde{\nabla }_{E_{\beta }}E_{\gamma }-\widetilde{\nabla }%
_{E_{\beta }}\widetilde{\nabla }_{E_{\alpha }}E_{\gamma }-\widetilde{\nabla }%
_{\left[ E_{\alpha },E_{\beta }\right] }E_{\gamma }
\end{equation*}%
with respect to the adapted frame. For the curvature tensor $\widetilde{R}$
of the metric connection $\widetilde{\nabla }$, with the help of (\ref{BA2.3}%
) and (\ref{BA3.7}), we have:

\begin{proposition}
\label{propo3} Let $\nabla $ be a torsion-free linear connection on $M$ and $%
T^{\ast }M$ be the cotangent bundle with the modified Riemann extension $%
\overline{g}_{\nabla ,c}$ over $(M,\nabla )$. The curvature tensor ${%
\widetilde{{}R}}$ of the metric connection ${\widetilde{{}\nabla }}$ on $%
T^{\ast }M$ satisfies the following conditions:%
\begin{eqnarray*}
\widetilde{R}(E_{i},E_{j})E_{k} &=&R_{ijk}^{\text{ \ \ \ }h}E_{h}{\,} \\
&&+\frac{1}{2}\{\nabla _{i}(\nabla _{k}c_{jh}-\nabla _{h}c_{jk})-\nabla
_{j}(\nabla _{k}c_{ih}-\nabla _{h}c_{ik}) \\
&&-R_{ijk}^{\text{ \ \ \ }m}c_{mh}-R_{ijh}^{\text{ \ \ \ }m}c_{km}\}E_{%
\overline{h}} \\
\widetilde{R}(E_{i},E_{j})E_{\overline{k}} &=&R_{jih}^{\text{ \ \ \ }k}E_{%
\overline{h}},{\,} \\
\widetilde{R}(E_{i},E_{\overline{j}})E_{k} &=&{\,0,}\text{ }\widetilde{R}%
(E_{i},E_{\overline{j}})E_{\overline{k}}={\,0,}\text{ }\widetilde{R}(E_{%
\overline{i}},E_{j})E_{k}=0, \\
\widetilde{R}(E_{\overline{i}},E_{j})E_{\overline{k}} &=&0,\text{ }%
\widetilde{R}(E_{\overline{i}},E_{\overline{j}})E_{k}=0,\text{ }\widetilde{R}%
(E_{\overline{i}},E_{\overline{j}})E_{\overline{k}}=0
\end{eqnarray*}%
with respect to the adapted frame $\{E_{\beta }\}$.
\end{proposition}

Let $\widetilde{X}$ and $\widetilde{Y}$ be vector fields of $T^{\ast }M$.
The curvature operator $\widetilde{R}(\widetilde{X},\widetilde{Y})$ is a
differential operator on $T^{\ast }M$. Similarly, for vector fields $X$ and $%
Y$ of $M$, $R(X,Y)$ is a differential operator on $M$. Now, we operate the
curvature operator $\widetilde{R}(\widetilde{X},\widetilde{Y})$ to the
curvature tensor $\widetilde{R}$. That is, for all $\widetilde{Z},\widetilde{%
W}$ and $\widetilde{U}$, we consider the condition $(\widetilde{R}(%
\widetilde{X},\widetilde{Y})\widetilde{R})(\widetilde{Z},\widetilde{W})%
\widetilde{U}=0$. In the case, we shall call the cotangent bundle $T^{\ast
}M $ as semi-symmetric with respect to the metric connection $\widetilde{%
\nabla }$.

In the adapted frame $\left\{ E_{\beta }\right\} $, the tensor $(\widetilde{R%
}(\widetilde{X},\widetilde{Y})\widetilde{R})(\widetilde{Z},\widetilde{W})%
\widetilde{U}$ is locally expressed as follows:%
\begin{eqnarray}
&&((\widetilde{R}(\widetilde{X},\widetilde{Y})\widetilde{R})(\widetilde{Z},%
\widetilde{W})\widetilde{U})_{\alpha \beta \gamma \theta \sigma }^{\text{ \
\ \ \ \ \ \ \ }\varepsilon }  \label{BA32.1} \\
&=&\widetilde{R}_{\alpha \beta \tau }^{\text{ \ \ \ }\varepsilon }\widetilde{%
R}_{\gamma \theta \sigma }^{\text{ \ \ \ }\tau }-\widetilde{R}_{\alpha \beta
\gamma }^{\text{ \ \ \ }\tau }\widetilde{R}_{\tau \theta \sigma }^{\text{ \
\ \ }\varepsilon }-\widetilde{R}_{\alpha \beta \theta }^{\text{ \ \ \ }\tau }%
\widetilde{R}_{\gamma \tau \sigma }^{\text{ \ \ \ }\varepsilon }-\widetilde{R%
}_{\alpha \beta \sigma }^{\text{ \ \ \ }\tau }\widetilde{R}_{\gamma \theta
\tau }^{\text{ \ \ \ }\varepsilon }.  \notag
\end{eqnarray}
Similarly, in local coordinates,%
\begin{eqnarray*}
&&((R(X,Y)R)(Z,W)U)_{ijklm}^{\text{ \ \ \ \ \ \ \ }n} \\
&=&R_{ijp}^{\text{ \ \ \ }n}R_{klm}^{\text{ \ \ \ \ }p}-R_{ijk}^{\text{ \ \
\ }p}R_{plm}^{\text{ \ \ \ \ }n}-R_{ijl}^{\text{ \ \ \ }p}R_{kpm}^{\text{ \
\ \ \ }n}-R_{ijm}^{\text{ \ \ \ }p}R_{klp}^{\text{ \ \ \ \ }n}.
\end{eqnarray*}

\begin{theorem}
\label{theore1}Let $\nabla $ be a torsion-free linear connection on $M$ and $%
T^{\ast }M$ be the cotangent bundle with the modified Riemannian extension $%
\overline{g}_{\nabla ,c}$ over $(M,\nabla )$. Under the assumption that $%
\nabla _{i}(\nabla _{k}c_{jh}-\nabla _{h}c_{jk})-\nabla _{j}(\nabla
_{k}c_{ih}-\nabla _{h}c_{ik})-R_{ijk}^{\text{ \ \ \ }m}c_{mh}-R_{ijh}^{\text{
\ \ \ }m}c_{km}=0$, where $R$ is the curvature tensor of $\nabla $, the
cotangent bundle $T^{\ast }M$ is semi-symmetric with respect to the metric
connection $\widetilde{\nabla }$ if and only if the base manifold $M$ is
semi-symmetric with respect to $\nabla $.
\end{theorem}

\begin{proof}
We consider the conditions $(\widetilde{R}(\widetilde{X},\widetilde{Y})%
\widetilde{R})(\widetilde{Z},\widetilde{W})\widetilde{U}=0$ for all vector
fields $\widetilde{X},\widetilde{Y},\widetilde{Z},\widetilde{W}$ and $%
\widetilde{U}$ on $T^{\ast }M$.

For all cases $\alpha =(i,\overline{i})$, $\beta =(j,\overline{j})$, $\gamma
=(k,\overline{k})$, $\theta =(l,\overline{l})$, $\sigma =(m,\overline{m})$
and $\varepsilon =(h,\overline{h})$ in (\ref{BA32.1}), the non-zero
components of the tensor $((\widetilde{R}(\widetilde{X},\widetilde{Y})%
\widetilde{R})(\widetilde{Z},\widetilde{W})\widetilde{U})_{\alpha \beta
\gamma \theta \sigma }^{\text{ \ \ \ \ \ \ \ \ }\varepsilon }$ are as
follows: 
\begin{eqnarray}
i) &&((\widetilde{R}(\widetilde{X},\widetilde{Y})\widetilde{R})(\widetilde{Z}%
,\widetilde{W})\widetilde{U})_{ijklm}^{\text{ \ \ \ \ \ \ \ \ }h}
\label{BA32.2} \\
&=&\widetilde{R}_{ijp}^{\text{ \ \ \ }h}\widetilde{R}_{klm}^{\text{ \ \ \ \ }%
p}+\widetilde{R}_{ij\overline{p}}^{\text{ \ \ \ }h}\widetilde{R}_{klm}^{%
\text{ \ \ \ \ }\overline{p}}-\widetilde{R}_{ijk}^{\text{ \ \ \ }p}%
\widetilde{R}_{plm}^{\text{ \ \ \ \ }h}-\widetilde{R}_{ijk}^{\text{ \ \ \ }%
\overline{p}}\widetilde{R}_{\overline{p}lm}^{\text{ \ \ \ \ }h}  \notag \\
&&-\widetilde{R}_{ijl}^{\text{ \ \ \ }p}\widetilde{R}_{kpm}^{\text{ \ \ \ \ }%
h}-\widetilde{R}_{ijl}^{\text{ \ \ \ }\overline{p}}\widetilde{R}_{k\overline{%
p}m}^{\text{ \ \ \ \ }h}-\widetilde{R}_{ijm}^{\text{ \ \ \ }p}\widetilde{R}%
_{klp}^{\text{ \ \ \ \ }h}-\widetilde{R}_{ijm}^{\text{ \ \ \ }\overline{p}}%
\widetilde{R}_{kl\overline{p}}^{\text{ \ \ \ \ }h}  \notag \\
&=&R_{ijp}^{\text{ \ \ \ }h}R_{klm}^{\text{ \ \ \ \ }p}-R_{ijk}^{\text{ \ \
\ }p}R_{plm}^{\text{ \ \ \ \ }h}-R_{ijl}^{\text{ \ \ \ }p}R_{kpm}^{\text{ \
\ \ \ }h}-R_{ijm}^{\text{ \ \ \ }p}R_{klp}^{\text{ \ \ \ \ }h}  \notag \\
&=&((R(X,Y)R)(Z,W)U)_{ijklm}^{\text{ \ \ \ \ \ \ \ }h}.  \notag \\
&&ii)\text{ \ \ \ }((\widetilde{R}(\widetilde{X},\widetilde{Y})\widetilde{R}%
)(\widetilde{Z},\widetilde{W})\widetilde{U})_{ijkl\overline{m}}^{\text{ \ \
\ \ \ \ \ \ }\overline{h}}  \notag \\
&=&\widetilde{R}_{ijp}^{\text{ \ \ \ }\overline{h}}\widetilde{R}_{kl%
\overline{m}}^{\text{ \ \ \ \ }p}+\widetilde{R}_{ij\overline{p}}^{\text{ \ \
\ }\overline{h}}\widetilde{R}_{kl\overline{m}}^{\text{ \ \ \ \ }\overline{p}%
}-\widetilde{R}_{ijk}^{\text{ \ \ \ }p}\widetilde{R}_{pl\overline{m}}^{\text{
\ \ \ \ }\overline{h}}-\widetilde{R}_{ijk}^{\text{ \ \ \ }\overline{p}}%
\widetilde{R}_{\overline{p}l\overline{m}}^{\text{ \ \ \ \ }\overline{h}} 
\notag \\
&&-\widetilde{R}_{ijl}^{\text{ \ \ \ }p}\widetilde{R}_{kp\overline{m}}^{%
\text{ \ \ \ \ }\overline{h}}-\widetilde{R}_{ijl}^{\text{ \ \ \ }\overline{p}%
}\widetilde{R}_{k\overline{p}\overline{m}}^{\text{ \ \ \ \ }\overline{h}}-%
\widetilde{R}_{ij\overline{m}}^{\text{ \ \ \ }p}\widetilde{R}_{klp}^{\text{
\ \ \ \ }\overline{h}}-\widetilde{R}_{ij\overline{m}}^{\text{ \ \ \ }%
\overline{p}}\widetilde{R}_{kl\overline{p}}^{\text{ \ \ \ \ }\overline{h}} 
\notag \\
&=&-R_{ijp}^{\text{ \ \ \ }m}R_{klh}^{\text{ \ \ \ \ }p}+R_{ijh}^{\text{ \ \
\ }p}R_{klp}^{\text{ \ \ \ \ }m}+R_{ijk}^{\text{ \ \ \ }p}R_{plh}^{\text{ \
\ \ \ }m}+R_{ijl}^{\text{ \ \ \ }p}R_{kph}^{\text{ \ \ \ \ }m}  \notag \\
&=&-((R(X,Y)R)(Z,W)U)_{ijklh}^{\text{ \ \ \ \ \ \ \ }m}.  \notag \\
&&iii)\text{ \ }((\widetilde{R}(\widetilde{X},\widetilde{Y})\widetilde{R})(%
\widetilde{Z},\widetilde{W})\widetilde{U})_{ijklm}^{\text{ \ \ \ \ \ \ \ \ }%
\overline{h}}  \notag \\
&=&\widetilde{R}_{ijp}^{\text{ \ \ \ }\overline{h}}\widetilde{R}_{klm}^{%
\text{ \ \ \ \ }p}+\widetilde{R}_{ij\overline{p}}^{\text{ \ \ \ }\overline{h}%
}\widetilde{R}_{klm}^{\text{ \ \ \ \ }\overline{p}}-\widetilde{R}_{ijk}^{%
\text{ \ \ \ }p}\widetilde{R}_{plm}^{\text{ \ \ \ \ }\overline{h}}-%
\widetilde{R}_{ijk}^{\text{ \ \ \ }\overline{p}}\widetilde{R}_{\overline{p}%
lm}^{\text{ \ \ \ \ }\overline{h}}  \notag \\
&&-\widetilde{R}_{ijl}^{\text{ \ \ \ }p}\widetilde{R}_{kpm}^{\text{ \ \ \ \ }%
\overline{h}}-\widetilde{R}_{ijl}^{\text{ \ \ \ }\overline{p}}\widetilde{R}%
_{k\overline{p}m}^{\text{ \ \ \ \ }\overline{h}}-\widetilde{R}_{ijm}^{\text{
\ \ \ }p}\widetilde{R}_{klp}^{\text{ \ \ \ \ }\overline{h}}-\widetilde{R}%
_{ijm}^{\text{ \ \ \ }\overline{p}}\widetilde{R}_{kl\overline{p}}^{\text{ \
\ \ \ }\overline{h}}  \notag
\end{eqnarray}

If we assume that%
\begin{eqnarray*}
\widetilde{R}_{ijk}^{\text{ \ \ \ }\overline{h}} &=&\nabla _{i}(\nabla
_{k}c_{jh}-\nabla _{h}c_{jk})-\nabla _{j}(\nabla _{k}c_{ih}-\nabla
_{h}c_{ik}) \\
&&-R_{ijk}^{\text{ \ \ \ }m}c_{mh}-R_{ijh}^{\text{ \ \ \ }m}c_{km}=0,
\end{eqnarray*}%
then it follows from (\ref{BA32.2}) that $(\widetilde{R}(\widetilde{X},%
\widetilde{Y})\widetilde{R})(\widetilde{Z},\widetilde{W})\widetilde{U}=0$ if
and only if $(R(X,Y)R)(Z,W)U=0$. This completes the proof.
\end{proof}

Denote by $\widetilde{R}_{\alpha \beta }=$ $\widetilde{R}_{\sigma \alpha
\beta }^{\text{ \ \ \ \ \ \ }\sigma }$ the contracted curvature tensor
(Ricci tensor) of the metric connection $\widetilde{\nabla }$. The only
non-zero component of $\widetilde{R}_{\alpha \beta }$ is as follows: $%
\widetilde{R}_{ij}=R_{ij}$, where $R_{ij}$ denote the components of the
Ricci tensor of $\nabla $ on $M$. Now we prove the following theorem.

\begin{theorem}
Let $\nabla $ be a torsion-free linear connection on $M$ and $T^{\ast }M$ be
the cotangent bundle with the modified Riemannian extension $\overline{g}%
_{\nabla ,c}$ over $(M,\nabla )$. The cotangent bundle $T^{\ast }M$ is Ricci
semi-symmetric with respect to the metric connection $\widetilde{\nabla }$
if and only if the base manifold $M$ is Ricci semi-symmetric with respect to 
$\nabla $.
\end{theorem}

\begin{proof}
The tensor $(\widetilde{R}(\widetilde{X},\widetilde{Y})\widetilde{Ric})(%
\widetilde{Z},\widetilde{W})$ has the components%
\begin{equation}
((\widetilde{R}(\widetilde{X},\widetilde{Y})\widetilde{Ric})(\widetilde{Z},%
\widetilde{W}))_{\alpha \beta \gamma \theta }=\widetilde{R}_{\alpha \beta
\gamma }^{\text{ \ \ \ }\varepsilon }\widetilde{R}_{\varepsilon \theta }+%
\widetilde{R}_{\alpha \beta \theta }^{\text{ \ \ \ }\varepsilon }\widetilde{R%
}_{\gamma \varepsilon }  \label{BA32.3}
\end{equation}%
with respect to the adapted frame $\{E_{\beta }\}$.

Choosing $\alpha =i,\beta =j,\gamma =k,\theta =l$ in (\ref{BA32.2}), we find%
\begin{eqnarray*}
((\widetilde{R}(\widetilde{X},\widetilde{Y})\widetilde{Ric})(\widetilde{Z},%
\widetilde{W}))_{ijkl} &=&\widetilde{R}_{ijk}^{\text{ \ \ \ }p}\widetilde{R}%
_{pl}+\widetilde{R}_{ijl}^{\text{ \ \ \ }p}\widetilde{R}_{kp} \\
&=&R_{ijk}^{\text{ \ \ \ }p}R_{pl}+R_{ijl}^{\text{ \ \ \ }p}R_{kp} \\
&=&((R(X,Y)Ric)(Z,W))_{ijkl},
\end{eqnarray*}%
all the others being zero. This finishes the proof.
\end{proof}

For the scalar curvature $\widetilde{r}$ of the metric connection ${%
\widetilde{{}\nabla }}$ with respect to $\overline{g}_{\nabla ,c}$, with the
help of (\ref{BA3.2}) we find%
\begin{equation*}
\widetilde{r}=\widetilde{R}_{\alpha \beta }(\overline{g}_{\nabla
,c})^{\alpha \beta }=0.
\end{equation*}%
Thus we have the following theorem.

\begin{theorem}
Let $\nabla $ be a torsion-free linear connection on $M$ and $T^{\ast }M$ be
the cotangent bundle with the modified Riemannian extension $\widetilde{g}%
_{\nabla ,c}$ over $(M,\nabla )$. The scalar curvature of the cotangent
bundle $T^{\ast }M$ with the metric connection ${\widetilde{{}\nabla }}$
with respect to $\widetilde{g}_{\nabla ,c}$ vanishes.
\end{theorem}

Next, we shall apply the differential operator $\widetilde{R}(\widetilde{X},%
\widetilde{Y})$ to the torsion tensor $\widetilde{T}$ of the metric
connection $\widetilde{\nabla }$.

\begin{theorem}
Let $\nabla $ be a torsion-free linear connection on $M$ and $T^{\ast }M$ be
the cotangent bundle with the modified Riemannian extension $\overline{g}%
_{\nabla ,c}$ over $(M,\nabla )$. Then $\widetilde{R}(\widetilde{X},%
\widetilde{Y}).\widetilde{T}=0$ for all vector fields $\widetilde{X}$ and $%
\widetilde{Y}$ on $T^{\ast }M$, where $\widetilde{T}$ is the torsion tensor
of the metric connection $\widetilde{\nabla }$ if and only if the base
manifold $M$ is semi-symmetric with respect to $\nabla $.
\end{theorem}

\begin{proof}
The differential operator $\widetilde{R}(\widetilde{X},\widetilde{Y})$
applied the torsion tensor $\widetilde{T}$ of the metric connection $%
\widetilde{\nabla }$ is in the form: 
\begin{eqnarray*}
&&((\widetilde{R}(\widetilde{X},\widetilde{Y})\widetilde{T})(\widetilde{Z},%
\widetilde{W}))_{\alpha \beta \gamma \theta }^{\text{ \ \ \ \ \ \ \ \ }%
\varepsilon } \\
&=&\widetilde{R}_{\alpha \beta \tau }^{\text{ \ \ \ }\varepsilon }\widetilde{%
T}_{\gamma \theta }^{\text{ \ \ \ }\tau }-\widetilde{R}_{\alpha \beta \gamma
}^{\text{ \ \ \ }\tau }\widetilde{T}_{\tau \theta }^{\text{ \ \ \ }%
\varepsilon }-\widetilde{R}_{\alpha \beta \theta }^{\text{ \ \ \ }\tau }%
\widetilde{T}_{\gamma \tau }^{\text{ \ \ \ }\varepsilon }
\end{eqnarray*}%
with respect to the adapted frame $\{E_{\beta }\}$. It follows immediately
that 
\begin{equation*}
\left\{ 
\begin{array}{c}
((\widetilde{R}(\widetilde{X},\widetilde{Y})\widetilde{T})(\widetilde{Z},%
\widetilde{W}))_{ijkl}^{\text{ \ \ \ \ }\overline{h}}=\widetilde{R}_{ijm}^{%
\text{ \ \ }\overline{h}}\widetilde{T}_{kl}^{\text{ \ }m}+\widetilde{R}_{ij%
\overline{m}}^{\text{ \ \ }\overline{h}}\widetilde{T}_{kl}^{\text{ \ }%
\overline{m}} \\ 
-\widetilde{R}_{ijk}^{\text{ \ \ }m}\widetilde{T}_{ml}^{\text{ \ }\overline{h%
}}-\widetilde{R}_{ijk}^{\text{ \ \ }\overline{m}}\widetilde{T}_{\overline{m}%
l}^{\text{ \ }\overline{h}}-\widetilde{R}_{ijl}^{\text{ \ \ }m}\widetilde{T}%
_{km}^{\text{ \ }\overline{h}}-\widetilde{R}_{ijl}^{\text{ \ \ }\overline{m}}%
\widetilde{T}_{k\overline{m}}^{\text{ \ }\overline{h}} \\ 
=p_{s}(R_{ijh}^{\text{ \ \ }m}R_{klm}^{\text{ \ \ }s}+R_{ijk}^{\text{ \ \ }%
m}R_{mlh}^{\text{ \ \ }s}+R_{ijl}^{\text{ \ \ }m}R_{kmh}^{\text{ \ \ }s}) \\ 
=-p_{s}((R(X,Y)R)(Z,W)U)_{ijklh}^{\text{ \ \ \ \ \ \ \ }s}, \\ 
otherwise=0%
\end{array}%
\right.
\end{equation*}%
which finishes the proof.
\end{proof}

On an $n-$dimensional Riemannian manifold $(M,g)$, it was defined a
generalized $(0,2)-$symmetric $Z$ tensor given by \cite{Mantica} 
\begin{equation*}
Z(X,Y)=Ric(X,Y)+\phi g(X,Y)
\end{equation*}%
for all vector fields $X$ and $Y$ on $M$, where where $\phi $ is an
arbitrary scalar function. Analogous to this definition, it may be locally
define generalized $\widetilde{Z}$ tensor on $(T^{\ast }M,\overline{g}%
_{\nabla ,c})$ with respect to the metric connection $\widetilde{\nabla }$
as follows:%
\begin{equation*}
\widetilde{Z}_{\alpha \beta }=\widetilde{R}_{\alpha \beta }+\widetilde{\phi }%
(\overline{g}_{\nabla ,c})_{\alpha \beta }.
\end{equation*}%
Putting the values of $\widetilde{R}_{\alpha \beta }$ and $\overline{g}%
_{\nabla ,c}$ in the above equation, we have the non-zero components%
\begin{eqnarray}
\widetilde{Z}_{ij} &=&R_{ij}+\widetilde{\phi }c_{ij},  \label{BA32.4} \\
\widetilde{Z}_{\overline{i}j} &=&\widetilde{\phi }\delta _{j}^{i},  \notag \\
\widetilde{Z}_{i\overline{j}} &=&\widetilde{\phi }\delta _{j}^{i}.  \notag
\end{eqnarray}%
We can state the following theorem.

\begin{theorem}
Let $\nabla $ be a torsion-free linear connection on $M$ and $T^{\ast }M$ be
the cotangent bundle with the modified Riemannian extension $\overline{g}%
_{\nabla ,c}$ over $(M,\nabla )$. The cotangent bundle $T^{\ast }M$ is $%
\widetilde{Z}$ semi-symmetric with respect to the metric connection $%
\widetilde{\nabla }$ if and only if the base manifold $M$ is Ricci
semi-symmetric with respect to $\nabla $.
\end{theorem}

\begin{proof}
The tensor ($\widetilde{R}(\widetilde{X},\widetilde{Y}).\widetilde{Z})(%
\widetilde{Z},\widetilde{W})$ has the components%
\begin{equation}
((\widetilde{R}(\widetilde{X},\widetilde{Y}).\widetilde{Z})(\widetilde{Z},%
\widetilde{W}))_{\alpha \beta \gamma \theta }=\widetilde{R}_{\alpha \beta
\gamma }^{\text{ \ \ \ }\varepsilon }\widetilde{Z}_{\varepsilon \theta }+%
\widetilde{R}_{\alpha \beta \theta }^{\text{ \ \ \ }\varepsilon }\widetilde{Z%
}_{\gamma \varepsilon }  \label{BA32.5}
\end{equation}%
with respect to the adapted frame $\{E_{\beta }\}$.

By choosing $\alpha =(i,\overline{i})$, $\beta =(j,\overline{j})$, $\gamma
=(k,\overline{k})$ and $\theta =(l,\overline{l})$ in (\ref{BA32.5}), in view
of (\ref{BA32.4}) we find the only non-zero component%
\begin{eqnarray*}
&&((\widetilde{R}(\widetilde{X},\widetilde{Y}).\widetilde{Z})(\widetilde{Z},%
\widetilde{W}))_{ijkl} \\
&=&\widetilde{R}_{ijk}^{\text{ \ \ \ }h}\widetilde{Z}_{hl}+\widetilde{R}%
_{ijk}^{\text{ \ \ \ }\overline{h}}\widetilde{Z}_{\overline{h}l}+\widetilde{R%
}_{ijl}^{\text{ \ \ \ }h}\widetilde{Z}_{kh}+\widetilde{R}_{ijl}^{\text{ \ \
\ }\overline{h}}\widetilde{Z}_{k\overline{h}} \\
&=&R_{ijk}^{\text{ \ \ \ }h}(R_{hl}+\widetilde{\phi }c_{hl})+\frac{1}{2}%
\{\nabla _{i}(\nabla _{k}c_{jh}-\nabla _{h}c_{jk}) \\
&&-\nabla _{j}(\nabla _{k}c_{ih}-\nabla _{h}c_{ik})-R_{ijk}^{\text{ \ \ \ }%
m}c_{mh}-R_{ijh}^{\text{ \ \ \ }m}c_{km}\}\widetilde{\phi }\delta _{l}^{h} \\
&&+R_{ijl}^{\text{ \ \ \ }h}(R_{kh}+\widetilde{\phi }c_{kh})+\frac{1}{2}%
\{\nabla _{i}(\nabla _{l}c_{jh}-\nabla _{h}c_{jl}) \\
&&-\nabla _{j}(\nabla _{l}c_{ih}-\nabla _{h}c_{il})-R_{ijl}^{\text{ \ \ \ }%
m}c_{mh}-R_{ijh}^{\text{ \ \ \ }m}c_{lm}\}\widetilde{\phi }\delta _{k}^{h} \\
&=&R_{ijk}^{\text{ \ \ \ }h}R_{hl}+R_{ijl}^{\text{ \ \ \ }h}R_{kh} \\
&=&(R(X,Y)Ric)_{ijkl},
\end{eqnarray*}%
from which the proof follows.
\end{proof}

\subsection{Conharmonic Curvature tensor on the cotangent bundle with
respect to the metric connection $\protect\widetilde{\protect\nabla }$}

We recall that the conharmonic curvature tensor $V$ on an $n-$dimensional
Riemannian manifold $(M,g)$ is defined as a $(4,0)-$tensor by the formula

\begin{equation*}
V_{ijkl}=R_{ijkl}-\frac{1}{n-2}\left[
R_{jk}g_{il}-R_{ik}g_{jl}-R_{jl}g_{ik}+R_{il}g_{jk}\right] ,
\end{equation*}%
where $R_{ijkl}$ and $R_{ij}$ are respectively the components of the
Riemannian curvature tensor and the Ricci tensor. The conharmonic curvature
tensor was first introduced by Ishii (see, \cite{Ishii}). A Riemanian
manifold whose conharmonic curvature tensor vanishes is called
conharmonically flat.

Analogous to the conharmonic curvature tensor with respect to a Levi--Civita
connection $\nabla $, it may be given the conharmonic curvature tensor $%
\widetilde{V}$ on $T^{\ast }M$ with respect to the metric connection $%
\widetilde{\nabla }$ as follows:%
\begin{equation*}
\widetilde{V}_{\alpha \beta \gamma \varepsilon }=\widetilde{R}_{\alpha \beta
\gamma \varepsilon }-\frac{1}{2(n-1)}\left[ \widetilde{R}_{\beta \gamma }(%
\overline{g}_{\nabla ,c})_{\alpha \varepsilon }-\widetilde{R}_{\alpha \gamma
}(\overline{g}_{\nabla ,c})_{\beta \varepsilon }-\widetilde{R}_{\beta
\varepsilon }(\overline{g}_{\nabla ,c})_{\alpha \gamma }+\widetilde{R}%
_{\alpha \varepsilon }(\overline{g}_{\nabla ,c})_{\beta \gamma }\right] .
\end{equation*}%
Next we prove the following theorem:

\begin{theorem}
Let $\nabla $ be a torsion-free linear connection on $M$ and $T^{\ast }M$ be
the cotangent bundle with the modified Riemannian extension $\widetilde{g}%
_{\nabla ,c}$ over $(M,\nabla )$. The cotangent bundle $T^{\ast }M$ is
locally conharmonically flat with respect to the metric connection $%
\widetilde{\nabla }$ if and only if the base manifold $M$ is Ricci flat with
respect to $\nabla $ and the components $c_{ij}$ of $c$ satisfy the condition%
\begin{equation*}
\nabla _{i}(\nabla _{k}c_{jh}-\nabla _{h}c_{jk})-\nabla _{j}(\nabla
_{k}c_{ih}-\nabla _{h}c_{ik})+R_{ijk}^{\text{ \ \ \ }m}c_{mh}-R_{ijh}^{\text{
\ \ \ }m}c_{km}=0,
\end{equation*}%
where $R_{ijk}^{\text{ \ \ \ }m}$ denote the components of the curvature
tensor $R$ of $\nabla $.
\end{theorem}

\begin{proof}
If the components of the curvature tensor $\widetilde{R}$ of the metric
connection $\widetilde{\nabla }$ on $T^{\ast }M$ satisfy the following
equations: 
\begin{subequations}
\begin{equation}
\widetilde{R}_{\alpha \beta \gamma \varepsilon }=\frac{1}{2(n-1)}\left[ 
\widetilde{R}_{\beta \gamma }(\overline{g}_{\nabla ,c})_{\alpha \varepsilon
}-\widetilde{R}_{\alpha \gamma }(\overline{g}_{\nabla ,c})_{\beta
\varepsilon }-\widetilde{R}_{\beta \varepsilon }(\overline{g}_{\nabla
,c})_{\alpha \gamma }+\widetilde{R}_{\alpha \varepsilon }(\overline{g}%
_{\nabla ,c})_{\beta \gamma }\right] ,  \label{BA33.1}
\end{equation}%
then $T^{\ast }M$ is said to be locally conharmonically flat with respect to
the metric connection $\widetilde{\nabla }$.

On lowering the upper index in the proposition \ref{propo3}, we obtain the
components of the $(0,4)-$curvature tensor of the metric connection $%
\widetilde{\nabla }$ as follows: 
\end{subequations}
\begin{equation*}
\left\{ 
\begin{array}{c}
\widetilde{R}_{ijkl}=+\frac{1}{2}\{\nabla _{i}(\nabla _{k}c_{jl}-\nabla
_{l}c_{jk})-\nabla _{j}(\nabla _{k}c_{il}-\nabla _{l}c_{ik}) \\ 
+R_{ijk}^{\text{ \ \ \ }m}c_{ml}-R_{ijl}^{\text{ \ \ \ }m}c_{km}\} \\ 
\widetilde{R}_{ijk\overline{l}}=R_{ijk}^{\text{ \ \ \ }l} \\ 
\widetilde{R}_{ij\overline{k}l}=R_{jil}^{\text{ \ \ \ }k} \\ 
otherwise=0.%
\end{array}%
\right.
\end{equation*}%
Putting the values of $\widetilde{R}_{\alpha \beta \gamma \varepsilon }$, $%
\widetilde{R}_{\alpha \beta }$ and $(\overline{g}_{\nabla ,c})_{\beta
\varepsilon }$ respectively in (\ref{BA33.1}), we have%
\begin{eqnarray}
&&\nabla _{i}(\nabla _{k}c_{jl}-\nabla _{l}c_{jk})-\nabla _{j}(\nabla
_{k}c_{il}-\nabla _{l}c_{ik})+R_{ijk}^{\text{ \ \ \ }m}c_{ml}-R_{ijl}^{\text{
\ \ \ }m}c_{km}  \label{BA33.2} \\
&=&\frac{1}{2(n-1)}(R_{jk}c_{il}-R_{ik}c_{jl}-R_{jl}c_{ik}+R_{il}c_{jk}) 
\notag
\end{eqnarray}%
\begin{equation}
R_{ijk}^{\text{ \ \ }l}=\frac{1}{2(n-1)}(R_{jk}\delta _{i}^{l}-R_{ik}\delta
_{j}^{l})  \label{BA33.3}
\end{equation}%
\begin{equation*}
-R_{jil}^{\text{ \ \ }k}=\frac{1}{2(n-1)}(R_{il}\delta _{j}^{k}-R_{jl}\delta
_{i}^{k}).
\end{equation*}%
Contraction $i$ and $l$ in (\ref{BA33.3}) gives%
\begin{eqnarray*}
R_{ljk}^{\text{ \ \ }l} &=&\frac{1}{2(n-1)}(R_{jk}\delta
_{l}^{l}-R_{lk}\delta _{j}^{l}) \\
R_{jk} &=&\frac{1}{2(n-1)}(nR_{jk}-R_{jk}) \\
R_{jk} &=&\frac{1}{2(n-1)}R_{jk}(n-1) \\
R_{jk} &=&0,
\end{eqnarray*}%
that is, the torsion-free linear connection $\nabla $ is Ricci flat. In the
case, from (\ref{BA33.2}), it follows that%
\begin{equation*}
\nabla _{i}(\nabla _{k}c_{jl}-\nabla _{l}c_{jk})-\nabla _{j}(\nabla
_{k}c_{il}-\nabla _{l}c_{ik})+R_{ijk}^{\text{ \ \ \ }m}c_{ml}-R_{ijl}^{\text{
\ \ \ }m}c_{km}=0.
\end{equation*}
\end{proof}

\section{The Schouten-van Kampen connection associated to the Levi-Civita
connection of the modified Riemannian extension}

The Schouten-Van Kampen connection has been introduced in \cite{Schouten}
for a study of non-holonomic manifolds. The Schouten-Van Kampen connection
associated to the Levi-Civita connection $\overline{\nabla }$ of the
modified Riemannian extension $\overline{g}_{\nabla ,c}$ and adapted to the
pair of distributions $(H,V)$ are defined by 
\begin{equation}
\overline{\nabla }^{\ast }{}_{\widetilde{X}}\widetilde{Y}=H(\overline{\nabla 
}_{\widetilde{X}}H\widetilde{Y})+V(\overline{\nabla }_{\widetilde{X}}V%
\widetilde{Y})  \label{BA4.1}
\end{equation}%
for all vector fields $\widetilde{X}$ and $\widetilde{Y}$, where $V$ and $H$
are the projection morphism of $TT^{\ast }M$ on $VT^{\ast }M$ and $HT^{\ast
}M$ respectively. The formula (\ref{BA4.1}) for $\overline{\nabla }^{\ast }$
has been first given by Ianus (see, \cite{Ianus}). By using (\ref{BA4.1})
and (\ref{BA3.3}), the Schouten-Van Kampen connection associated to the
Levi-Civita connection $\overline{\nabla }$ of the modified Riemannian
extension $\overline{g}_{\nabla ,c}$ are locally given by the following
formulas:%
\begin{equation*}
\left\{ 
\begin{array}{c}
\overline{\nabla }_{E_{\overline{i}}}^{\ast }E_{\overline{j}}=0,\text{ }%
\overline{\nabla }_{E_{\overline{i}}}^{\ast }E_{j}=0, \\ 
\overline{\nabla }_{E_{i}}^{\ast }E_{\overline{j}}=-\Gamma _{ih}^{j}E_{%
\overline{h}},\text{ }\overline{\nabla }_{E_{i}}^{\ast }E_{j}=\Gamma
_{ij}^{h}E_{h},%
\end{array}%
\right. 
\end{equation*}%
which are the components of the horizontal lift $^{H}\nabla $ of the
torsion-free linear connection $\nabla $. Hence we get:

\begin{proposition}
\label{propo2}Let $\nabla $ be a torsion-free linear connection on $M$ and $%
T^{\ast }M$ be the cotangent bundle with the modified Riemannian extension $%
\overline{g}_{\nabla ,c}$ over $(M,\nabla )$. The Schouten-Van Kampen
connection $\overline{\nabla }^{\ast }$ associated to the Levi-Civita
connection $\overline{\nabla }$ of the modified Riemannian extension $%
\overline{g}_{\nabla ,c}$ and the horizontal lift $^{H}\nabla $ of the
torsion-free linear connection $\nabla $ to $T^{\ast }M$ coincide to each
other.
\end{proposition}

In view of Proposition \ref{propo1}, Proposition \ref{propo2}, Theorem \ref%
{theore1} and its proof, it immediately follows the final result.

\begin{theorem}
Let $\nabla $ be a torsion-free linear connection on $M$ and $T^{\ast }M$ be
the cotangent bundle with the modified Riemannian extension $\overline{g}%
_{\nabla ,c}$ over $(M,\nabla )$. The cotangent bundle $T^{\ast }M$ is
semi-symmetric with respect to the Schouten-Van Kampen connection $\overline{%
\nabla }^{\ast }$ associated to the Levi-Civita connection $\overline{\nabla 
}$ of the modified Riemannian extension $\overline{g}_{\nabla ,c}$ if and
only if the base manifold $M$ is semi-symmetric with respect to $\nabla $.
\end{theorem}

\bigskip

\end{document}